\documentclass[12pt]{article}

\usepackage{amsmath,amsfonts,amssymb,latexsym,amsthm,verbatim,a4wide}

 \theoremstyle{plain}
  \newtheorem{theorem}{Theorem}[section]
  \newtheorem{corollary}[theorem]{Corollary}
  \newtheorem{proposition}[theorem]{Proposition}
  \newtheorem{lemma}[theorem]{Lemma}  
  \newtheorem{definition}[theorem]{Definition}
  \newtheorem{conjecture}[theorem]{Conjecture}
  \newtheorem{remark}[theorem]{Remark}

\newcommand{\pr}{{\mathbb {P}}}

\newcommand{\re}{{\mathbb{R}}}
\newcommand{\vc}[1]{{\mathbf #1}}
\newcommand{\blah}[1]{}

\newcommand{\bern}{\mathcal{B}}

\newcommand{\binomi}[2]{{\rm Bin}(#1, #2)}

\newcommand{\wtu}{\widetilde{u}}

\newcommand{\ol}[1]{\overline{#1}} 
\newcommand{\LL}{{\mathcal{L}}}

\title{A proof of the Shepp--Olkin entropy concavity conjecture}
\author{Erwan Hillion\thanks{Department of Mathematics, University of Luxembourg, 
Campus Kirchberg, L-1359 Luxembourg, 
Grand-Duchy of Luxembourg. Email {\tt erwan.hillion@uni.lu}} \and Oliver Johnson\thanks{School of Mathematics,  University of Bristol, 
University Walk, 
Bristol BS8 1TW, 
 United Kingdom.  
Email {\tt o.johnson@bristol.ac.uk}}}
\date{\today}

\begin{document}

\maketitle 

\begin{abstract}
\noindent We prove the Shepp--Olkin  conjecture, which states that the entropy of the sum of independent
Bernoulli random variables is concave in the parameters of the individual random variables. Our proof is a refinement of
an argument previously presented by the same authors, which resolved the conjecture in the monotonic case (where all the 
parameters are simultaneously increasing). In fact, we show that the monotonic case is the worst case, using
 a careful analysis of concavity properties of the derivatives of the probability mass function. We propose a generalization
of Shepp and Olkin's original conjecture, to consider R\'{e}nyi and Tsallis entropies.
\end{abstract}



\section{Introduction}

This paper considers a conjecture of Shepp and Olkin \cite{shepp}, that the entropy of Bernoulli sums is a concave function
of the parameters. We write $\bern(p)$ for the law of a Bernoulli variable with parameter $p$.
Let $(p_1,\ldots, p_n) \in [0,1]^n$ be a $n$-tuple of parameters, and consider independent 
random variables $(X_1, \ldots, X_n)$ with $X_i \sim \bern(p_i)$. We set $S = \sum_{i=1}^n X_i$ and, for $k \in \{0,\ldots n\}$,
we write
 $f_k := \pr(S=k)$ for the probability mass function of $S$, defining a probability measure supported on $\{0,\ldots,n\}$. 
Note that $S$ is sometimes referred to as having a Poisson binomial distribution.

For each $k \in \{0,\ldots,n\}$, the probabilities $f_k$ can be seen as a smooth function of the $n$ parameters $\vc{p} := (p_1,\ldots,p_n)$. For instance we have $f_0 = (1-p_1) \cdots (1-p_n)$ and $f_n = p_1 \cdots p_n$.
 A particular case is obtained when the parameters $p_1=\cdots=p_n=p$ are all equal. In this case, $(f_k)$ describes the binomial measure $\binomi{n}{p}$.
This paper is focused on the study of the Shannon entropy $H$ of $(f_k)$.

\begin{definition}
Writing $U(x) := x \log(x)$ if $x>0$ and $U(0) := 0$, we define:
\begin{equation} \label{eq:shannon}
H(\vc{p}) := - \sum_{k=0}^n U(f_k) :=  - \sum_{k=0}^n f_k \log(f_k).
\end{equation}
\end{definition}

The entropy $H$ can itself be seen as a smooth function of the parameters $p_1,\ldots, p_n$. This article is devoted to the proof of the following:

\begin{theorem}[Shepp-Olkin Theorem] \label{th:SO}
For any $n \geq 1$, the function $\vc{p}  \mapsto H(\vc{p})$ is concave.
\end{theorem}

We  simplify notation somewhat by considering the case where each $p_i := p_i(t)$ is an affine function of 
parameter $t \in [0,1]$, so that the derivative $p_i' := \frac{d}{dt} p_i(t)$ is constant in $t$. 
Theorem \ref{th:SO} will follow if we can show that the entropy is   concave in $t$.

Theorem \ref{th:SO}
 was conjectured by Shepp and Olkin \cite{shepp}  in 1981. In their original paper, Shepp and Olkin stated that the
conjecture is true in the cases $n=2$ and $n=3$, and proved that it holds for the binomial case where all $p_i$ are identical
(see also Mateev \cite{mateev}). Since then, progress has been limited. In \cite{johnsonc6}, Yu and Johnson considered
the thinning operation of R\'{e}nyi \cite{renyi4}, and proved a result which implies concavity of entropy when each $p_i(t)$
is proportional to $t$ or $1-t$. Further, Hillion \cite[Theorem 1.1]{hillion2}
proved Theorem \ref{th:SO} in the case where each $p_i(t)$ is either constant
or equal to $t$.

More significant progress was made  in \cite{johnson34} by the present authors, who proved that the entropy is concave when
 all $p_i'$ have the same sign. Perhaps surprisingly, in the current paper we show that this `monotone'
case resolved in \cite{johnson34}  is the most difficult.  

The strategy of \cite{johnson34} was to show that the entropy is a concave function along an interpolating (geodesic) 
path between discrete random variables, assuming
certain conditions are satisfied. In particular, \cite[Theorem 4.4]{johnson34} showed that if a `$k$-monotonicity' condition,
a generalized log-concavity condition, and so-called Condition 4 hold, then the entropy is concave. At the heart of this analysis
was control of terms of the form $U(1-x)$ by the second-order Taylor expansion of $U$. This relied on the fact that a term
we refer to as $B_k$ (see \eqref{eq:abc}) satisfies $B_k \leq 0$.

However,  unfortunately the inequality $B_k \leq 0$ does not hold in general, so in this paper we will use
Lemma \ref{lem:UtotalMonotone} to control terms of this kind.
 The strategy  is essentially to show that Condition 4 (here referred to as Proposition \ref{prop:hupper}) holds in the general Shepp-Olkin case.
Using this result, we  deduce Corollary \ref{cor:fgh} which plays the part of the $k$-monotonicity and generalized
log-concavity conditions. For the sake of simplicity we restrict our discussion to mass functions which are Bernoulli sums, but
 a version of Theorem \ref{th:SO}
will hold for interpolating paths made up of log-concave $f$ for which Proposition \ref{prop:hupper} holds.

The study of the entropy in the monotone case \cite{johnson34} was motivated by the theory of transportation of discrete probability measures
 and the introduction of a distance defined by a discrete form of the Benamou--Brenier formula
\cite{benamou2}  (see also \cite{gozlan}). This idea of a discrete geodesic was designed as an analogue to the more
developed theory of geodesics on continuous spaces such as Riemannian manifolds (see for example \cite{villani}), where concavity
of the entropy relates to the Ricci curvature of the underlying manifold (see \cite{carlen4,lott,sturm,sturm2}). Further, using
ideas related to the Bakry-\'{E}mery $\Gamma$-calculus \cite{bakry}, concavity of the entropy can be used to prove functional inequalities,
such as log-Sobolev or HWI (see for example \cite{ane,cordero3}). It remains an important problem to provide discrete
analogues of this theory.

The structure of the paper is as follows: in Section \ref{sec:technical} we state some technical results required for the proof of Theorem \ref{th:SO}, the proofs of which 
are deferred to the appendix. In Section \ref{sec:proof} we prove Theorem \ref{th:SO} itself. In Section \ref{sec:tsallis} we propose a generalized form of 
Shepp and Olkin's conjecture, in terms of  R\'{e}nyi and Tsallis  entropies $H_{R,q}$ and $H_{T,q}$. 

\section{Technical results required in the proof} \label{sec:technical}

We now state a number of technical results which are required in the proof of Theorem \ref{th:SO}, the main result of the paper.
The proofs are deferred to the appendix.

\subsection{Concavity of functions}
We first state a technical result concerning certain functions $U$:

\begin{lemma}\label{lem:UtotalMonotone}
Let $U: (0,\infty) \mapsto \re$ be a function such that (i) $U(1) = 0$, (ii) $U'(1) = 1$, (iii) $U'''(t) \leq 0$ for all $t$
 and (iv) $\log U''(t)$ is convex in $t$.

For $A,B,C,\alpha,\beta,\gamma$ satisfying $0<A<1, 0<C<1$ and $B^2 \leq AC$, $\beta^2 \leq \alpha \gamma$ we have:
\begin{equation} \label{eq:UtotalMonotone}
\alpha U(1-A)-2 \beta U(1-B)+\gamma U(1-C) \geq -\alpha A + 2 \beta B - \gamma C.
\end{equation}
\end{lemma}

Note that the conditions of this lemma are satisfied for $U(x) = x \log x$.

\subsection{Cubic inequality for Bernoulli sums}

Let $S=X_1+\ldots +X_m$ be the sum of independent Bernoulli variables of parameters $p_1,\ldots, p_m$.
\begin{proposition} \label{prop:BernoulliInequality}
For any $m$ and $k$, the following inequality holds:
\begin{equation} \label{eq:2foldlog}
\pi_{k-2} \pi_{k+1}^2 + \pi_{k}^3 + \pi_{k-1}^2 \pi_{k+2} \geq \pi_{k-2} \pi_{k} \pi_{k+2} + 2 \pi_{k-1} \pi_{k} \pi_{k+1}.
\end{equation}
\end{proposition}
Note that for $k$ outside the support of $S$, each term is equal to zero, so the inequality is trivially true.
In the appendix, we show that
 Proposition \ref{prop:BernoulliInequality} can be proved directly, using 
two other cubic inequalities \eqref{eq:c1} and \eqref{eq:c1bar} taken from \cite{johnson34}.

If we write $D_k := \LL(\pi_k) = \pi_k^2 - \pi_{k-1} \pi_{k+1} $, then it is well known (see Proposition \ref{prop:logConc}) that
$D_k \geq 0$, a result referred to as log-concavity of $\pi_k$. Observe that \eqref{eq:2foldlog} is equivalent to the statement
that $\LL^2(\pi_k) = \LL(D_{k}) = D_{k}^2 - D_{k-1} D_{k+1} \geq 0$, that the $D_k$ themselves are log-concave,
a property is referred to as 2-fold log-concavity.

This result also follows from a result of Br\"{a}nd\'{e}n \cite{branden}, which resolved a conjecture made
independently by  Stanley, by  McNamara and Sagan and by Fisk (see Br\"{a}nd\'{e}n's paper \cite{branden} for details).
   Br\"{a}nd\'{e}n discusses conditions under which infinite log-concavity (the fact that the iterated
$\LL^r(\pi_k) \geq 0$ for all $r \geq 1$) holds. 

\subsection{Upper bounding $h_k$}

Recall that we consider the random variable $S = \sum_{l=1}^n X_l$ with probability mass function $f_k(t)$.
For $i \in \{1,\ldots,n \}$ we define $S^{(i)} := \sum_{l \neq i} X_l$,
with probability mass function $f^{(i)}_k$ supported on $\{ 0, \ldots, n-1 \}$.
Similarly, for a pair of indices $i \neq j \in \{1,\ldots, n\}$ we define $S^{(i,j)} := \sum_{l \notin \{i,j\}} X_l$ with mass
function $f^{(i,j)}_k$ supported on $\{ 0, \ldots, n-2 \}$.

\begin{definition} As in \cite{johnson34}, we make the following definitions:
\begin{eqnarray} \label{eq:deffgh}
g_k &:=& \sum_{i} p_i' f^{(i)}_k := \sum_{i} p_i' \pr(S^{(i)} =k), \\
h_k &:=& \sum_{i \neq j} p_i' p_j' f^{(i,j)}_k := \sum_{i \neq j} p_i' p_j' \pr(S^{(i,j)}=k).   \label{eq:deffgh2}
\end{eqnarray}
\end{definition}
We prove a strong upper bound on $h_k$, which lies at the heart of the proof of Theorem \ref{th:SO}:
\begin{proposition} \label{prop:hupper}
For $k= 0, \ldots, n-2$:
\begin{equation} \label{eq:NewCondition4}
h_k (f_{k+1}^2-f_k f_{k+2}) \leq 2 g_k g_{k+1} f_{k+1} - g_k^2 f_{k+2} - g_{k+1}^2 f_k.
\end{equation}
\end{proposition}
\begin{corollary} \label{cor:fgh} For $k= 0, \ldots, n-2$, the $h_k f_k \leq g_k^2$ and $h_kf_{k+2} \leq g_{k+1}^2$.\end{corollary}
In \cite{johnson34}, Proposition \ref{prop:hupper} is referred to as Condition 4, and is proved for the case where all $p_i'$ have the same sign.
In fact, we show  using Proposition \ref{prop:BernoulliInequality} that this inequality holds in general.

\section{Proof of the Shepp--Olkin conjecture} \label{sec:proof}

Theorem~\ref{th:SO} is obvious if $n=1$. We now fix some $n \geq 2$ and 
take $X_1,\ldots X_n$ to be independent Bernoulli variables with parameters $p_1,\ldots,p_n$, where each $p_i = p_i(t)$ is an affine function of $t$ (with constant derivative $p_i'$).

\begin{proposition} \label{prop:SOuk}
To prove Theorem~\ref{th:SO}, it suffices to show that $u_k \geq 0$ for any choice of parameters $(p_1, \ldots, p_n)$, of slopes $(p_1', \ldots, p_n')$ and index $k \in \{0,\ldots,n-2 \}$, where we write:
\begin{equation} \label{eq:defu}
u_k := h_k \log\left( \frac{f_k f_{k+2}}{f_{k+1}^2} \right) + \left( \frac{g_k^2}{f_k} - 2 \frac{g_k g_{k+1}}{f_{k+1}} + \frac{g_{k+1}^2}{f_{k+2}} \right).
\end{equation}
\end{proposition}
\begin{proof}
As in \cite{johnson34}, direct calculation (for example using the probability generating function) shows that the first two derivatives
of  $f_k$ satisfy
\begin{eqnarray}
\frac{d f_k}{dt}(t) & = &  g_{k-1} - g_k , \label{eq:fder} \\
\frac{d^2 f_k}{dt^2}(t) & = &  h_{k} - 2 h_{k-1} + h_{k-2}. \label{eq:fder2}
\end{eqnarray}
Hence, we can write the derivative of the entropy: 
\begin{eqnarray}
H''(t) &  = & - \sum_{k=0}^n \frac{d^2}{dt^2} U(f_k(t)) \nonumber \\
& = & - \sum_{k=0}^n U''(f_k) \left( \frac{d f_k}{dt} \right)^2 - \sum_{k=0}^n U'(f_k) \left( \frac{d^2 f_k}{dt^2} \right) 
\nonumber \\
& = & - \sum_{k=0}^n U''(f_k) \left( g_{k-1} - g_k \right)^2 - \sum_{k=0}^n U'(f_k)  \left(   h_{k} - 2 h_{k-1} + h_{k-2}  \right) 
\label{eq:drop2} \\
& \leq & - \sum_{k=0}^{n-2} \biggl[ \left( g_{k}^2 U''(f_k) - 2 g_k g_{k+1} U''(f_{k+1})  + g_{k+1}^2 U''(f_{k+2}) \right)  \nonumber \\
&  & \hspace*{1.5cm}
+ h_k \left( U'(f_k) - 2 U'(f_{k+1}) + U'(f_{k+2} ) \right) \biggr].  \label{eq:drop}
\end{eqnarray}
The form of Equation (\ref{eq:drop2}) follows using \eqref{eq:fder} and \eqref{eq:fder2}.
The relabelling in \eqref{eq:drop} uses the fact that $g_k$ is supported on $ \{ 0, \ldots, n-1 \}$ and $h_k$ supported
on $\{ 0, \ldots, n-2 \}$. This expression is an inequality since (taking into account the end points of
the range of summation) we remove the terms
$- g_{n-1}^2 U''(f_{n-1}) - g_0^2 U''(f_1)$, which are negative assuming that $U'' \geq 0$.

Making the choice of $U(x) = x \log x$ we deduce the form of $u_k$ given in \eqref{eq:defu}, since $U''(x) = 1/x \geq 0$ in this
case.
\end{proof}

One of the main differences with the monotonic case studied in \cite{johnson34} is that the quantities $(g_l)$ and $(h_l)$ are not necessarily positive. However, we note the following:

\begin{lemma}
If $h_k \leq 0$ then $u_k  \geq 0$. 
\end{lemma}
\begin{proof}
This follows quite easily from the log-concavity property for the $(f_l)$, see Proposition~\ref{prop:logConc}, which implies that $h_k \log\left( \frac{f_k f_{k+2}}{f_{k+1}^2} \right) \geq 0$. On the other hand, using the arithmetic-geometric inequality and then again the log-concavity property, we have:
\begin{equation*} 
\frac{g_k^2}{f_k} + \frac{g_{k+1}^2}{f_{k+2}} \geq  2 \frac{|g_k| |g_{k+1}|}{ \sqrt{f_k f_{k+2}}} \geq 2 \frac{|g_k| |g_{k+1}|}{ f_{k+1}} \geq 2 \frac{g_k g_{k+1}}{f_{k+1}},
\end{equation*}
which proves that $u_k \geq 0$.
\end{proof}

Hence, from now on, we can suppose that $h_k>0$. 

\begin{proof}[Proof of Theorem \ref{th:SO}] 
We set:
\begin{equation} \label{eq:abc}
A_k := \frac{g_k^2-f_kh_k}{g_k^2} \ , \ B_k := \frac{|g_k| |g_{k+1}|-f_{k+1}h_k}{|g_k| |g_{k+1}|} \ , \ C_k := \frac{g_{k+1}^2-f_{k+2}h_k}{g_{k+1}^2},
\end{equation}
and:
\begin{equation} \label{eq:alphaetc}
\alpha_k := \frac{g_k^2}{f_k} \ , \ \beta_k := \frac{|g_k||g_{k+1}|}{f_{k+1}} \ , \ \gamma_k := \frac{g_{k+1}^2}{f_{k+2}}.
\end{equation}
Noticing that $ \log(g_k^2) + \log(g_{k+1}^2) - 2 \log(|g_k| |g_{k+1}|) = 0, $ we have: 
\begin{equation}
\log\left( \frac{f_k f_{k+2}}{f_{k+1}^2} \right) = \log\left(\frac{h_k f_k}{g_k^2} \right) - 2 \log\left(\frac{h_k f_{k+1}}{|g_k| |g_{k+1}|} \right)+\log\left(\frac{h_k f_{k+2}}{g_{k+1}^2} \right).
\end{equation}
Straightforward calculations give:
\begin{equation} \label{eq:ukTransform}
u_k \geq \alpha_k U(1-A_k)-2 \beta_k U(1-B_k)+\gamma_k U(1-C_k) + \left( \alpha_k A_k - 2 \beta_k B_k + \gamma_k C_k \right).
\end{equation}

 We need to verify the assumptions of
Lemma \ref{lem:UtotalMonotone} on $A_k$, $B_k$, $C_k$.
First,  we can show that $B_k^2 \leq A_k C_k$, which is equivalent to
\begin{equation}
h_k (f_{k+1}^2-f_k f_{k+2}) \leq 2 |g_k||g_{k+1}| f_{k+1} - g_k^2 f_{k+2} - g_{k+1}^2 f_k.
\end{equation}
However, we notice that this is implied by the stronger inequality \eqref{eq:NewCondition4} from Proposition \ref{prop:hupper}.
Second, we need to check the
 the inequalities $A_k \geq 0$ and $C_k \geq 0$. As we have  proved that $0 \leq B_k^2 \leq A_k C_k$, it suffices to verify that $A_k \geq 0$, which is a  restatement of Corollary \ref{cor:fgh}.

Clearly, the log-concavity of $f$ shows that $\beta_k^2 \leq \alpha_k \gamma_k$ for the quantities defined in
\eqref{eq:alphaetc}
We can thus apply Lemma \ref{lem:UtotalMonotone} to deduce that $u_k \geq 0$ and prove Theorem~\ref{th:SO}. 
\end{proof}

\section{Tsallis and R\'{e}nyi entropy} \label{sec:tsallis}

Having resolved the Shepp-Olkin conjecture for (Shannon) entropy, it is natural to want to generalize our result to a wider
class of entropy-like functionals. Recall the following definitions, each of which reduce to the Shannon entropy
\eqref{eq:shannon} as $q \rightarrow 1$.

\begin{definition} Given a probability mass function $f$ supported on $\{ 0, \ldots, n \}$, for $0 \leq q \leq \infty$  define
\begin{eqnarray}
\mbox{1. $q$-R\'{e}nyi entropy (see \cite{renyi2}): \;\;}
\label{eq:renyi} H_{R,q}(f) &  = & \frac{1}{1-q} \log \left( \sum_{x=0}^n f_x^q \right), \\
\mbox{2. $q$-Tsallis entropy (see \cite{tsallis}): \;\;}
 \label{eq:tsallis} H_{T,q}(f) & = & \frac{1}{q-1}  \left( 1-  \sum_{x=0}^n f_x^q \right).\end{eqnarray}
\end{definition}

We know that  $H_{R,\infty}(f) = -\log \max_x f(x)$ (the min-entropy), which is not concave for Bernoulli sums (e.g. for Bernoulli($p$) with $p < 1/2$, this
is just $-\log (1- p)$, which is convex in $p$).
For $q=0$, R\'{e}nyi entropy is the log of the size of the support, which is constant (and hence concave) for Bernoulli sums. 
This suggests the following conjecture:
\begin{conjecture}[Generalized Shepp-Olkin conjecture] \mbox{ } \label{conj:gSO}
\begin{enumerate}
\item There is a critical 
$q_R^*$ such that the $q$-R\'{e}nyi entropy of all Bernoulli sums is concave for $q \leq q_R^*$, and the entropy of some interpolation is convex for $q > q^*_R$. 
\item There is a  critical 
$q_T^*$ such that the $q$-Tsallis entropy of all Bernoulli sums is concave for $q \leq q_T^*$, and the entropy of some interpolation is convex for $q > q^*_T$. 
\end{enumerate}
Indeed (based on Lemma \ref{lem:values} below) we conjecture that $q_R^* = 2$ and $q_T^* = 3.65986\ldots$, the root of $2 - 4 q + 2^q = 0$.
\end{conjecture}
We mention some limited progress towards this conjecture.
\begin{lemma} \label{lem:values} \mbox{ }
\begin{enumerate}
\item 
 For any $q > 2$, there exists a Shepp--Olkin interpolation with convex $q$-R\'{e}nyi entropy. 
\item
 For any $q > q^* := 3.65986 \ldots$, there exists a Shepp--Olkin interpolation with convex $q$-Tsallis entropy.
\end{enumerate}
\end{lemma}
\begin{proof}  \mbox{ } 
\begin{enumerate}
\item
Consider the Bernoulli $\bern(p)$ family, for which $T(p) = p^q + (1-p)^q$. As $p \rightarrow 0$, since $q > 2$, the $T(p) \rightarrow 1$.
Similarly, $T(p)' = q( p^{q-1} + (1-p)^{q-1}) \rightarrow q$, and $T(p)'' = q(q-1) (p^{q-2} + (1-p)^{q-2}) \rightarrow q(q-1)$.  In Equation (\ref{eq:renyider2}) we obtain
$-q - \frac{q^2}{1-q} = q/(q-1) > 0$.

\item  Consider the Binomial $(2,p)$ family, for which $T(p) = (p^2)^q + (2p(1-p))^q + (1-p)^{2q}$.
The second derivative of $H_{T,q}$ at $p = 1/2$ is $2^{3-2q}  ( 2 -  4 q + 2^q) q/(q-1)$, which is positive for $q > q^*$.
\end{enumerate}
\end{proof}
Although R\'{e}nyi and Tsallis entropies are monotone
functions of one another, and so are maximised by the same mass function, the relationship between their concavity properties involves the chain rule. 
\begin{lemma} \mbox{ }
\begin{enumerate}
\item For $q < 1$, if the Tsallis entropy is concave, then so is the R\'{e}nyi entropy.
\item For $q > 1$, if the R\'{e}nyi entropy is concave, then so is the Tsallis entropy.
\end{enumerate}
\end{lemma}
\begin{proof}
If we write $T(t) = \sum_x f_x(t)^q$ for some path, then
\begin{equation} \label{eq:renyider2}
H''_{R,q}(t) = \frac{T''(t)}{(1-q)T(t)} - \frac{1}{1-q} \left( \frac{T'(t)}{T(t)} \right)^2 = \frac{1}{T(t)} H''_{T,q}(f(t)) 
 - \frac{1}{1-q} \left( \frac{T'(t)}{T(t)} \right)^2. \end{equation}
Since the difference has a sign we can control, we can deduce the result holds. \end{proof}

\begin{remark}
We can consider the Tsallis entropy in the framework used earlier. Using Equation \eqref{eq:drop} we deduce that $H_{T,q}''(t) = - q \sum_{k=0}^{n-2} u_k$, where
\begin{eqnarray*}
u_k &:=& - \frac{1}{1-q} h_k (f_k^{q-1}-2f_{k+1}^{q-1}+f_{k+2}^{q-1}) +
 \left(g_k^2 f_k^{q-2} - 2 g_kg_{k+1}f_{k+1}^{q-2} + g_{k+1}^2 f_{k+2}^{q-2} \right).
\end{eqnarray*}
Conjecture \ref{conj:gSO} would follow if $u_k \geq 0$, or even if $\wtu_k := u_k + \nabla_1(v_k) \geq 0$, where $v_k$ is some function
and $\nabla_1$ represents the left discrete derivative.  \end{remark}

There is an heuristic argument which supports such a conjecture, at least in the monotonic case. Using the fact that 
$$
\left(g_k^2 f_k^{q-2} - 2 g_kg_{k+1}f_{k+1}^{q-2} + g_{k+1}^2 f_{k+2}^{q-2} \right) = 
(g_k-g_{k+1})^2 f_{k+1}^{q-2} + \nabla_1\left(g_{k+1}^2 \nabla_1(f_{k+2}^{q-2})\right),$$
by taking $v_k = -  \frac{1-q}{2-q} g_{k+1}^2 \nabla_1( f_{k+2}^{q-2})$ we can write:
\begin{equation}
\wtu_k = -\frac{1}{1-q} h_k \nabla_2(f_{k+2}^{q-1}) + (\nabla_1 g_{k+1})^2 f_{k+1}^{q-2}+\frac{1}{2-q} \nabla_1\left(g_{k+1}^2 \nabla_1(f_{k+2}^{q-2})\right),
\end{equation}
where $\nabla_1$ and $\nabla_2$ stand for the first and second left discrete derivatives. A possible continuous analogy consists in considering, at least at a formal level, the expression
\begin{equation} \label{eq:uTsallis}
u := -\frac{1}{1-q} h (f^{q-1})'' + (g')^2 f^{q-2} + \frac{1}{2-q} \left(g^2 (f^{q-2})'\right)',
\end{equation}
for a triple $(f,g,h)$ of real functions. If we make the further assumption that $g=vf$ and $h=v^2 f$ for some velocity function $v$, which was already an assumption made in the heuristic study of the monotonic case of the Shepp-Olkin conjecture 
\cite[Section 2]{johnson34} and referred to as a Benamou-Brenier condition, equation~\eqref{eq:uTsallis} is simplified into
\begin{equation*}
u = v'^2 f^q,
\end{equation*}
which is clearly non-negative, and corresponds to \cite[Corollary 2.6]{johnson34}.

\appendix

\section{Proof of technical results}

\subsection{Proof of functional inequality, Lemma \ref{lem:UtotalMonotone}}

\begin{proof}[Proof of Lemma \ref{lem:UtotalMonotone}]
We consider the function 
\begin{equation}
\xi(t) := \alpha U(1-tA)-2 \beta U(1-tB)+\gamma U(1-tC),
\end{equation}
well-defined and smooth for $0 \leq t \leq 1$. Its derivative is given by 
\begin{equation}
\xi'(t) = -A \alpha U'(1-tA) +2  B \beta U'(1-tB) - C \gamma U'(1-tC),
\end{equation}
so Inequality \eqref{eq:UtotalMonotone} can be rewritten $\xi(1) - \xi(0) \geq \xi'(0)$.
Since, by the mean value theorem, $\xi(1) - \xi(0) = \xi'(s)$ for some $0 \leq s \leq 1$,
it is  sufficient to show the convexity of $\xi$, i.e. that $\xi''(t) \geq 0$ for any $0<t<1$. We have:
\begin{eqnarray}
\xi''(t) & = & \alpha A^2 U''(1-tA) - 2\beta B^2 U''(1-t B) + \gamma C^2 U''(1-t C) \\
& \geq & 2 \sqrt{ \alpha \gamma A^2 C^2 U''(1-tA) U''(1-t C) } 
- 2\beta B^2 U''(1-t B)   \label{eq:amgm} \\
& \geq & 2 \beta B^2 \left( \sqrt{ U''(1-tA) U''(1-t C) } 
- U''(1-t B) \right)   \label{eq:step2} \\
& \geq & 2 \beta B^2 \left( \sqrt{ U''(1-tA) U''(1-t C) } 
- U''(1-t (A+C)/2) \right)   \label{eq:step3}
\end{eqnarray}
here (\ref{eq:amgm}) follows by the arithmetic mean-geometric mean inequality,
and (\ref{eq:step2}) follows by the assumptions $\beta \leq \alpha \gamma$ and $B^2 \leq AC$, and
(\ref{eq:step3}) uses the fact that by assumption (iii) $U''(s)$ is decreasing in $s$.

The result follows since we can deduce the positivity of \eqref{eq:step3} using assumption (iv) (the log-convexity of $
U''$)
\end{proof}

\subsection{Proof of Proposition \ref{prop:BernoulliInequality}}

For $\pi_k$ the probability mass function of the sum $T$, we
 note that the quadratic Newton inequality (see for example Niculescu \cite{niculescu}) gives the log-concavity of $\pi$:
\begin{proposition} \label{prop:logConc} 
For any $k \in \{0, \ldots, m-2 \}$ we have $\pi_{k+1}^2 \geq \pi_{k} \pi_{k+2}$.
\end{proposition}
Further properties of Bernoulli sums have been proven by the authors in \cite{johnson34}.
The most interesting for our purposes are the inequalities stated there as $C_1(k) \geq 0$ and $\ol{C}_1(k) \geq 0$, which give that:
\begin{eqnarray}
\left( \pi_{k-1}^2 - \pi_{k-2} \pi_{k} \right) \pi_{k+1} \leq \pi_{k-1} \left( \pi_{k}^2 - \pi_{k-1} \pi_{k+1} \right) \label{eq:c1} \\
\left( \pi_{k+1}^2 - \pi_{k} \pi_{k+2} \right) \pi_{k-1} \leq \pi_{k+1} \left( \pi_{k}^2 - \pi_{k-1} \pi_{k+1}
\right) \label{eq:c1bar} 
\end{eqnarray}
These results have been stated (and proven by induction on the number $n$ of parameters) as \cite[Theorem A2, Corollary
A3]{johnson34}.

Multiplying Equations \eqref{eq:c1} and \eqref{eq:c1bar} together, and rearranging, we simply obtain the positivity of
$$\pi_{k-1} \pi_{k} \pi_{k+1}  \left(
\pi_{k-2} \pi_{k+1}^2 + \pi_{k}^3 + \pi_{k-1}^2 \pi_{k+2} - \pi_{k-2} \pi_{k} \pi_{k+2} - 2 \pi_{k-1} \pi_{k} \pi_{k+1}
\right),$$
and we deduce Proposition \ref{prop:BernoulliInequality} holds.

\subsection{Proof of Proposition \ref{prop:hupper} }

In the monotonic case studied in \cite{johnson34}, Equation~\eqref{eq:NewCondition4} was referred to as Condition 4, and was
verified under the assumption that the $p_i'$ are non-negative. More precisely, 
some involved manipulations (using the definitions of $g_k$ and $h_k$ given in Equations~\eqref{eq:deffgh} 
and \eqref{eq:deffgh2}) allow us to deduce that (see Proposition 6.1  and Equation (73) of \cite{johnson34}):
\begin{lemma} \label{lem:NewCondition4Quadratic}
We can write the term 
\begin{eqnarray}
\lefteqn{
 \left( 2 g_k g_{k+1} f_{k+1} - g_k^2 f_{k+2} - g_{k+1}^2 f_k \right) - h_k (f_{k+1}^2-f_k f_{k+2}) } \nonumber \\
& = & \sum_{i < j}  \left({p_i'}^2 p_j(1-p_j) b_{i,j} + {p_j'}^2 p_i(1-p_i) b_{j,i} + 2 p_i'p_j' p_i(1-p_i)p_j(1-p_j) c_{i,j} \right),
\label{eq:quadratic}
\end{eqnarray}
where the coefficients satisfy
\begin{eqnarray} \label{eq:defcij}
c_{i,j} & := &  - \left(f_{k}^{(i,j)}\right)^3 + 2f_{k-1}^{(i,j)}f_{k}^{(i,j)}f_{k+1}^{(i,j)}-\left(f_{k+1}^{(i,j)}\right)^2 f_{k-2}^{(i,j)}-\left( f_{k-1}^{(i,j)}\right)^2f_{k+2}^{(i,j)}+f_{k-2}^{(i,j)}f_{k}^{(i,j)}f_{k+2}^{(i,j)} \;\;\;\;
\end{eqnarray}
\end{lemma}
Lemma 6.2 of \cite{johnson34} shows that
\begin{equation}
b_{i,j}  \geq   - \frac{1}{2} (p_i (1-p_j) + p_j (1-p_i)) c_{i,j} \label{eq:defbij}
\end{equation}
This result can be verified using the expression for $b_{i,j}$ given in \cite[Equation (72)]{johnson34}, using Equations
(\ref{eq:c1}) and (\ref{eq:c1bar}) and other related cubic inequalities for $\pi$.

We now observe that  $c_{i,j} \leq 0 $ by Proposition \ref{prop:BernoulliInequality},  simply by taking $\pi_k = f_{k}^{(i,j)}$ in Equation
\eqref{eq:2foldlog}. Combining this with \eqref{eq:defbij}, we deduce that $b_{i,j}$ is positive, and that (treated as a 
quadratic in $p'_i$ and $p_j'$, the bracketed term in Equation (\ref{eq:quadratic} has negative discriminant
\begin{eqnarray*}
4 p_i (1-p_i) p_j (1-p_j) c_{i,j}^2 - 4  b_{i,j} b_{j,i} 
& \leq & c_{i,j}^2 \left( 4 p_i (1-p_i) p_j (1-p_j) - (p_i (1-p_j) + p_j(1-p_i))^2 \right) \\
& = & - c_{i,j}^2 (p_i - p_j)^2,
\end{eqnarray*}
meaning that it is positive for all values of $p_i'$ and $p_j'$.

Note that the negativity of $c_{i,j}$ shows that if we fix $|p_i'|$ and $|p_j'|$ then (\ref{eq:quadratic}) is minimized when
$p_i'$ and $p_j'$ have the same sign, justifying the claim that the monotonic case is the worst case.

\subsection{Proof of Corollary \ref{cor:fgh}}

\begin{proof}[Proof of Corollary \ref{cor:fgh}] Proposition \ref{prop:hupper} gives an upper bound on $h_k$: 
\begin{equation}
h_k \left( f_{k+1}^2-f_k f_{k+2} \right) \leq 2 g_k g_{k+1} f_{k+1} - g_k^2 f_{k+2} - g_{k+1}^2 f_k.
\end{equation} In order to prove that 
$f_k h_k \leq g_k^2$, it thus suffices to show that
\begin{equation}
\left(2 g_k g_{k+1} f_{k+1} - g_k^2 f_{k+2} - g_{k+1}^2 f_k \right) f_k \leq \left(f_{k+1}^2-f_k f_{k+2}\right) g_k^2.
\end{equation}
But this equation can be simplified into 
\begin{equation}
\left(f_{k+1} g_k - f_k g_{k+1} \right)^2 \geq 0,
\end{equation}
which is obviously true. 
\end{proof}

\end{document}